\documentclass{article}

\usepackage[utf8]{inputenc}
\usepackage{amsfonts}
\usepackage{amsmath}
\usepackage{amsthm}
\usepackage{blindtext}
\usepackage{graphicx}
\usepackage[numbers]{natbib}
\usepackage{amssymb}
\usepackage{mathtools}
\usepackage{stmaryrd}
\usepackage{tikz-cd}
\usepackage{mathrsfs}

\newtheorem{theorem}{Theorem}[section]
\newtheorem{lemma}{Lemma}[section]
\newtheorem{corollary}{Corollary}[section]

\newtheorem{proposition}{Proposition}[section]
\newtheorem{proposition-definition}{Proposition-Definition}[section]
\theoremstyle{definition}
\newtheorem{definition}{Definition}[section]

\title{Endomorphisms of a certain ring of rational functions}
\author{by Milo Moses}
\begin{document}
\maketitle

\tableofcontents

\newcommand{\Vect}{\mathrm{Vect}}
\newcommand{\RR}{\mathbb{R}}
\newcommand{\HH}{\mathbb{H}}
\newcommand{\NN}{\mathbb{N}}
\newcommand{\QQ}{\mathbb{Q}}
\newcommand{\CC}{\mathbb{C}}
\newcommand{\ZZ}{\mathbb{Z}}
\newcommand{\LL}{\mathcal{L}}
\newcommand{\MM}{\mathcal{M}}
\newcommand{\MMsemi}{\mathcal{M}^{\mathrm{semi}}}
\newcommand{\Aut}{\mathrm{Aut}}
\newcommand{\rad}{\mathrm{rad}}
\newcommand{\Hom}{\mathrm{Hom}}
\newcommand{\Fpbar}{\overline{\mathbb{F}}_{p}}
\newcommand{\Fp}{\mathbb{F}_{p}}
\newcommand{\Gal}{\mathrm{Gal}}
\newcommand{\im}{\mathrm{im}\,}
\newcommand{\Frob}{\mathrm{Frob}}
\newcommand{\End}{\mathrm{End}}
\newcommand{\modd}{\,\,\mathrm{mod}}
\newcommand{\lcm}{\mathrm{lcm}}
\newcommand{\st}{\,\,\mathrm{st}\,\,}

\section{Introduction}

For any field $K$, define the ring $\MM_K$ as follows:

\begin{equation*}
\begin{pmatrix}
  \text{elements:  ratios of monic polynomials in $K(x)$}\\
  \text{addition: multiplication of rational functions}\\
  \text{multiplication: }\left[\prod_{i}(x-\alpha_i)\right]\cdot \left[\prod_{j}(x-\beta_j)\right]=\left[\prod_{i,j}(x-\alpha_i\beta_j)\right]
\end{pmatrix}
\end{equation*}

This sort of ring can naturally arise in the study of vector spaces. Endomorphisms of vector spaces have associated characteristic polynomials, and the direct sum/tensor product of endomorphisms corresponds to addition/multiplication in $\MM_K$. We will put a topology on $\MM_K$ by letting Cauchy sequences be those in which the leading coefficients converge. We will later show that it makes $\MM_K$ into a topological ring.  The following is the main theorem proved in this text:

\begin{theorem}\label{main} Let $K/\QQ$ be an algebraic extension, and let $\phi: \MM_K\to \MM_K$ be a continuous ring endomorphism. Then, either

\begin{enumerate}
\item There exist integers $k,s\geq 0$ and $\sigma\in \Gal(K/\QQ)$ such that

$$\phi\left(\prod_{i}(x-\alpha_i)\right)=x^{\sum_{i}\alpha_i^s-1}\cdot \prod_{i}\left(x-\sigma(\alpha_i)^k\right)$$

\item All roots of rational functions in the image of $\phi$ are roots of unity. If $K/\QQ$ is finite, then there exists $N$ such that all the roots are $N$th roots of unity.
\end{enumerate}
\end{theorem}

The constant $N$ appearing in part (2) of Theorem \ref{main} is only effective when $K=\QQ$. In that case one can easily compute $N$ given (say,) $\phi([x-2])$ using the methods of the proof. This can be turned into precise results, such as $N<\left(\# \text{ roots of } \phi([x-2])\right)^{1.8}$.
\footnote{The result here is not proved explicitly, but it is a simple corollary of our methods. The constant 1.8 is not optimal, and can be readily improved to $1/\ln(2)+\epsilon$}

\section{General Methods}

In this section, we will establish some general results about $\MM_K$, with $K$ being an arbitrary field. Once convenient, we will specialize to $K\subseteq \CC$. Our main tool will be the ring morphisms $t_k: \MM_K\to K$ defined on elements of the form $[x-\alpha]$ by

$$t_k([x-\alpha])=\alpha^k,$$

and extended to general elements by linearity. Note that we use the convention $0^0=1$, so $t_0$ is the degree map. A first property of these maps is that they can be used to compute polynomial coefficients:

\begin{theorem}[Newton's Formulas]
\label{Newton's Formulas}
For any polynomial $f\in K[x]$ of degree $d$ and $k\leq d$, we have that

$$ke_k(f)=\sum_{n=1}^k(-1)^{n+1}e_{k-n}(f)t_n([f]).$$

Here, $e_n(f)$ are the elementary symmetric polynomial coefficients defined implicitly by

$$f(x)=\sum_{n=0}^{d}(-1)^{d-n}e_{d-n}(f)x^n.$$

\end{theorem}
\begin{proof} These are classical.
\end{proof}

\begin{corollary}
\label{Tail}
If $K$ has characteristic $0$, then for any monic polynomials $f,g\in K[x]$, $e_k(f)=e_k(g)$ for all $k<N$ if and only of $t_k([f])=t_k([g])$ for all $k<N$.
\end{corollary}
\begin{proof} This is immediate from the Theorem.
\end{proof}

In particular, we have that the $t_k$ define an injection

$$\MM_K\xhookrightarrow{\prod_{k=0}^{\infty}t_k} \prod_{k=0}^{\infty}K.$$

This allows us to define the topology on $\MM_K$ as the subspace topology, with $\prod_{k=1}^{\infty}K$ given the product topology, and $K$ always given the discrete topology. It is clear that convergence in $\MM_K$ is equivalent to the convergence of the $t_k$, which in turn is equivalent to the convergence of the $e_k$, which are the leading coefficients of a polynomial. We will now show to what extent $t_k$ are the only continuous ring morphisms $\MM_K\to K$.

\begin{lemma}
\label{Homs Lemma}
If $S$ is a subring of $K^N$ for some natural number $N$, then the nonzero morphisms in $\Hom(S,K)$ are projections of $S$ onto a single index, followed by composition with some ring morphism $K\to K$.
\end{lemma}
\begin{proof} We proceed by induction on $N$. If $N=1$, then $S$ is a subring of $K$. Thus, the map $S\to K$  is automatically projection of $S$ onto a single index followed by a morphism. Additionally, it is standard application of Zorn's lemma that every map $S\to K$ can be lifted to a map $K\to K$.

For our inductive step, suppose the result is proved for $N-1$. Fix a map $\phi: S\to \ZZ$. It is clear that $\phi$ doesn't factor through the projection $L\to L'$ forgetting the $1$st index if and only if there are two elements which are equal outside of the $1$st index, but give different elements under $\phi$. Equivalently, this means that there is a nonzero element $(a_1,0,0...)$ on which $\phi$ does not act by $0$.

If $\phi$ factors through the projection forgetting the $1$st or $2$nd index, then the induction hypothesis gives us the conclusion. Suppose it does not factor through forgetting either. This means there are elements $(a_1,0,0...)$ and $(0,a_2,0...)$ under which $\phi$ does not act by $0$. However, these elements multiply to $0$. This implies there exist nonzero elements of $K$ which multiply to $0$, which is a contradiction. This completes the proof.
\end{proof}

\begin{proposition}
\label{Homs}
The only continuous ring morphisms $\phi:\MM_{K}\to K$ are of the form $\sigma\circ t_k$, where $\sigma:K\to K$ is a ring morphism.
\end{proposition}
\begin{proof} Since $K$ is discrete, the kernel of $\phi$ contains an open neighborhood of $0$. The ideals

$$I_N=\left\{r\st t_k(r)=0\,\,\forall \,\,k\geq N\right\}$$

form a basis of open neighborhoods of $0$ under the topology for $\MM_K$. Hence, we have that $I_N\in \ker\phi$ for some $N$. Hence, any continuous map $\MM_K\to K$ factors through projection onto $\MM_K/I_N$. This is a subring of $K^N$, hence Lemma \ref{Homs Lemma} gives the desired result.
\end{proof}

This is extremely powerful, as it lets us define the following:

\begin{definition} For any continuous endomorphisms $\phi: \MM_K\to \MM_K$, we define the functions $\phi^*: \ZZ_{\geq 0}\to \ZZ_{\geq 0}$ and $\sigma^{\phi}: \ZZ_{\geq 0}\to \left(\text{morphisms }K\to K\right)$ to be the unique values such that

$$t_n(\phi(r))=\sigma^{\phi}_n(t_{\phi^*(n)}(r))$$

for all $r\in \MM_K$.
\end{definition}

We now define the most important endomorphisms of $\MM_K$:

\begin{proposition} For every $n\geq 1$, the map

$$\lambda_n: \MM_K \to \MM_K$$

sending $\prod_{i}(x-\alpha_i)$ to $\prod_{i}(x-\alpha_i^n)$ is a continuous endomorphism of $\MM_K$. We have that $\lambda_k^*(n)=kn$ for all $k,n$, and $\sigma_n^{\lambda_k}=1$.
\end{proposition}
\begin{proof} This is a straightforward computation.
\end{proof}

The following formula will underpin our use of cyclotomic methods throughout the proof.

\begin{corollary} If $K$ has characteristic $0$, then for any $f\in \overline{K}[x]$,

$$x^df(1/x)=\prod_{n\geq1}\left(1-x^{n}\right)^{c_n}$$

where

$$c_n=\frac{1}{n}\sum_{d|n}\mu(n/d)t_d([f]).$$

The quantities $(1+x)^z$ are defined using Newton's generalized binomial theorem.

\end{corollary}
\begin{proof}
This is the so-called Cyclotomic identity, and can be found in many references.
\end{proof}

Let $\MM_\ZZ$ be the subring of $\MM_\QQ$ consisting of rational functions with integer coefficients.

\begin{proposition}
\label{Mob Inv}
For any $r\in \MM_{\ZZ}$ and $n\in \ZZ_{\geq 1}$,

$$\frac{1}{n}\sum_{d|n}\mu(n/d)t_d(r)\in \ZZ.$$
\end{proposition}
\begin{proof} We prove the result for $r=[f]$, $f\in \ZZ[x]$ monic, and extend to a general element by linearity. Let $N$ be the smallest number for which $c_N=\frac{1}{N}\sum_{d|N}\mu(n/d)t_d(r)$ is not in $\ZZ$. Now, we look at the coefficient of $x^N$ on both sides of

$$x^df(1/x)=\prod_{n\geq1}\left(1-x^{n}\right)^{c_n}.$$

The left hand side's coefficient is an element of $\ZZ$ by assumption. On the right hand side, terms with $n>N$ do not contribute since they are too large. By assumption, all terms with $n<N$ contribute an element of $\ZZ$. Thus, the final answer is an element of $R$ if and only if the coefficient of $x^N$ in $\left(1-x^{N}\right)^{c_N}$ is in $\ZZ$. The coefficient is exactly $-c_N$ and hence $c_N\in \ZZ$ which is a contradiction. This yields the result.
\end{proof}

The following lemmas are necessary to continue:

\begin{proposition}
\label{Periodic}
Let $s_k$ be any sequence of complex numbers given explicitly by

$$s_k=\sum_{j}c_j \theta_j^{k}$$

where $j$ runs over a finite set of values, $c_j$ are complex, and $\theta_j$ are on the unit circle. There are infinitely many $q\geq 1$ such that

$$|s_{k}-s_{k+q}|\leq \frac{4\pi}{q^{1/J}}\sum_{i}|c_i|$$

for every $k\geq 0$, where $J$ is the number of values $j$ runs over. For every $k_0$, $s_{k_0}$ is a limit point of $s_k$. In particular, if $s_k$ converges then it must be constant. The only way for $s_k$ to be constant is for the coefficients of all roots other than $1$ to be zero.
\end{proposition}
\begin{proof}
To begin, we assume without loss of generality that all the $\theta_j$ are distinct. Let $\theta_j=e^{2\pi i x_j}$. By Dirichlet's Simultaneous Approximation theorem, we can find infinitely many values $q$ for which there exists $p_j$ such that

$$\left|x_j-\frac{p_j}{q}\right|\leq\frac{1}{q^{1+1/J}}.$$

Using the estimate $\left|e^x-1\right|\leq 2x$ for $\left|x\right|<1$, we get

\begin{align*}
\left|\exp\left(2\pi i q x_j\right)-1\right|\leq \frac{4\pi}{q^{1/J}}.
\end{align*}

Choosing any $k\geq 1$ and summing over $j$, this immediately gives

\begin{align*}
\left|s_{k}-s_{k+q}\right|\leq\frac{4\pi}{q^{1/J}}\sum_{i}\left|c_i\right|.
\end{align*}

Since this tends to $0$ as $q$ gets large, it is clear that every value of $s_k$ must be a limit point.

Now, suppose that  $s_k$ is a constant. By subtracting off the coefficient of $1$, we can assume $s_k=0$ uniformly. The average value of $|s_k|^2$ is $\sum_j |c_j|^2$, and hence we must have $c_j=0$ for all $j$, as desired.
\end{proof}

Sometimes the following will also be useful:

\begin{proposition}
\label{finite periodic}
If $r\in \MM_\CC$ is such that $t_k(r)$ goes over finitely many values, then all roots of $r$ are roots of unity.
\end{proposition}
\begin{proof} If $r$ has any roots greater than $1$, then Proposition \ref{Periodic} gives that $t_k(r)$ tends to infinity. Hence, we must have that all roots of $r$ are in absolute value less than or equal to $1$. We thus have

$$r=\prod_{i}(x-\theta_i)^c_i\cdot r',$$

where $\theta_i$ are distinct complex numbers of absolute value $1$, and all roots of $r'$ are less than $1$. Letting $E<1$ be the largest root of $r'$, we get that

$$t_k(r)=S_k+O(E^n)$$

where $S_k=\sum_i c_i \cdot \theta_i^k$, and $O$ denotes big-Oh notation. Choose $q$ large so that $|S_{k+q}-S_k|$ is uniformly small, which is possible by Proposition \ref{Periodic}. Then, we would have that

$$t_{k+q}(r)-t_{k}(r)$$

is uniformly small as $q$ gets large, for large enough $k$. However, since $t_k(r)$ goes over finitely values, this must imply that $t_k(r)=t_{k+q}(r)$ for all large enough $k$. In particular, we must have that $S_{k+q}-S_k$ grows small at an exponentially small rate. Proposition \ref{Periodic} says this means that $S_k$ is uniformly zero. Hence, we get that $\theta_i^q=1$ for all $q$.

Now, consider the roots of $r$ smaller than $1$. Since $S_k$ is periodic, we must have that the $t_k$ of the smaller roots must go over finitely many values. However, the $t_k$ of the smaller roots also tend to zero. The only way to go over finitely many values and tend to zero is if the $t_k$ is zero for all large enough roots. The only way to be zero for large enough roots is if there are no roots of absolute value less than $1$, by repeating the same as above.
\end{proof}

The above work motivates the following definition:

\begin{proposition}
\label{MZZ}
The following conditions on an element $r\in \MM_{\ZZ}$ are equivalent:

\begin{enumerate}
\item $r=[f/g]$, where the roots of $f$ and $g$ are all of absolute value $1$ or $0$
\item $r=[f/g]$, where the roots of $f$ and $g$ are all of absolute value $\leq 1$
\item $\lim_{k\to\infty}t_k(r)/\left(1+\epsilon\right)^k=0$ for all $\epsilon>0$
\item $r=c_0[x]+\sum_{n\geq 1}c_n\left[1-x^{-n}\right]$, where $c_n\in \ZZ$ are constants only finitely of which are non-zero
\end{enumerate}

The subring of elements $r\in \MM_\ZZ$ with any (all) of these equivalent properties is called $\MM_{\ZZ,\Phi}$.
\end{proposition}
\begin{proof} We begin by showing that $2\implies 1$. To do this, we note that the smallest non-zero coefficient of $f$ is equal to the product of its non-zero roots. This product must be a non-zero integer, and since none of the roots are larger than $1$ they must all be of absolute value $1$.

Now we show that $3\implies 2$. Let $C$ be the absolute value of the largest root of $r$. If $C\leq 1$, we are done. If $C>1$, write

$$t_n(r)=S_n\cdot C^n+O(E^n)$$,

where $S_n=\sum_i c_i\cdot \theta_i^n$, $|\theta_i|=1$, and $E<C$. Then,

$$S_n=\frac{t_n(r)}{C^n}+O((E/C)^n),$$

so $S_n\to 0$. Thus $S_n$ meets the conditions of Lemma \ref{Periodic}, and hence is zero constant. This contradicts the fact that there is at least one root of absolute value $C$, since otherwise the sum would be non-trivial.

Now we show that $4\implies 3$. To do this, we note that

$$t_k\left(c_0[x]+\sum_{n\geq 1}c_n\left[1-x^{-n}\right]\right)$$

is periodic and hence tends to zero when divided by arbitrarily large values.

Finally we show that $1\implies 4$. We induct on $\deg(f)+\deg(g)$. If $\deg(f)\geq1$ with root $\zeta$, then $\zeta$  is an algebraic number all of whose Galois conjugates are also of absolute value one and hence must be a primitive root of unity of degree $n$, $\zeta=\zeta_n$, by a classical theorem of Kronecker. Thus, $\Phi_{n}(x)|\deg(f)$ where $\Phi_n$ is the $n$th cyclotomic polynomial and hence

$$[f/g]=\left[\left(f/\Phi_{n}(x)\right)/g\right]+\sum_{d|n}\mu(n/d)\left[x^d-1\right].$$

Continuing until the only roots of $f$ and $g$ are all zero, we are done.
\end{proof}

We define $\widehat{\MM_K}$ to be the topological completion of $\MM_K$.

\begin{proposition}
\label{tk evaluation}
The subring $\MM_{\ZZ,\Phi}$ is dense in $\MM_{\ZZ}$. Every element $r\in\widehat{\MM_{\ZZ,\Phi}}=\widehat{\MM_{\ZZ}}$ can be written uniquely in the form

$$r=c_0[x]+\sum_{n \geq 1}c_n\left[1-x^{-n}\right],$$

for arbitrary $c_n$. Moreover, for $k\in \ZZ_{\geq 1}$

$$t_k(r)=\sum_{d|k}c_d$$

and $t_0(r)=c_0$.
\end{proposition}
\begin{proof} To begin, we show that every element of the form $r=c_0[x]+\sum_{n \geq 1}c_n\left[1-x^{-n}\right]$ is in $\widehat{\MM_{\ZZ,\Phi}}$. That follows from the fact that

$$r=\lim_{N\to\infty}\left(c_0[x]+\sum_{n=1}^{N}c_n\left[1-x^{-n}\right]\right),$$

where each term in the limit is in $\MM_{\ZZ,\Phi}$ by Proposition \ref{MZZ} part 3. Next, we show that these sums follow the $t_k$ formula stated. Namely, we begin by checking

\begin{align*}
t_k\left([x^{n}-1]\right)&=\sum_{j=0}^{n-1}e^{2\pi ji/k}\\
&=
\begin{cases} n & n|k\\
0 & \mathrm{otherwise}.
\end{cases}
\end{align*}

Hence, dividing through by $x^{n}$, we get the same formula for $t_k\left([1-x^{-n}]\right)$ except that it has no contribution at $0$. Extending by linearity we get the formula. Uniqueness follows from Mobius inversion.

For the last part of the proposition, we check that every element of $\widehat{\MM_{\ZZ}}$ really can be written in this form. let $r_N$ be a sequence in $\MM_{\ZZ}$ with $r$ as a limit point. Let

$$c_n=\lim_{N\to\infty}\frac{1}{n}\sum_{d|n}\mu(n/d)t_d(r_N),$$

which converges since $t_k$ are continuous. Moreover, $c_n\in \ZZ$ since by Proposition \ref{Mob Inv} they are at every step in the limit. It is now clear that $r=c_0[x]+\sum_{n\geq 1}c_n[1-x^{-n}]$ by Möbius inversion.
\end{proof}

We now prove an important multiplication formula on $\MM_{\ZZ,\Phi}$:

\begin{proposition}
\label{Multiplication Lemma}
For any $n,m$ it holds that

$$\left[x^{n}-1\right]\left[x^{m}-1\right]=\gcd(n,m)\left[x^{\lcm(n,m)}-1\right]$$
\end{proposition}
\begin{proof} This is a trivial computation that follows directly from the definition of multiplication in $\MM_{\ZZ}$.
\end{proof}

\section{Endomorphisms of $\MM_\QQ$}

In this section, we study $\End\left(\MM_{\QQ}\right)$. A natural place to start is to study maps on the smaller domain $\MM_{\ZZ,\Phi}$:

\begin{proposition}
\label{MZZ Endos}

If $\phi\in \MM_{\ZZ,\Phi}\to \MM_\CC$, is continuous then for all rational prime powers $p^{\alpha}$ we have that

$$\phi\left(\left[x^{p^{\alpha}}-1\right]\right)=s_{-\infty}[x]+\sum_{\beta=0}^{\alpha}s_\beta \cdot p^{\alpha-\beta}\cdot \left[x^{p^{\beta}}-1\right],$$

where $s_{\beta}\in\{-1,0,1\}$.

\end{proposition}
\begin{proof} Seeing as $t_k(\phi(r))$ must go over finitely values, Proposition \ref{finite periodic} gives us that $\phi(r)$ lies in $\MM_{\ZZ,\Phi}$. Let $r=\phi\left(\left[x^{p^{\alpha}}-1\right]\right)$, and let $c_n=\frac{1}{n}\sum_{d|n}\mu(n/d)t_d(r)$. Choose $N$ such that $c_N\neq 0$. Write $N=p^{\gamma}N'$, with $p$ not dividing $N'$. We begin by showing that $\gamma\leq \alpha$. Namely we compute

\begin{align*}
p^{\gamma}N'\leq p^{\gamma}N'|c_N|&=\left|\sum_{d|p^{\gamma}N'}\mu(p^{\gamma}N'/d)t_d(r)\right|\\
&=\left|\sum_{d|N'}\mu(N'/d)\left(t_{p^{\gamma}d}(r)-t_{p^{\gamma-1}d}(r)\right)\right|\\
&\leq \sum_{d|N'}p^{\alpha}\leq p^{\alpha}N'.
\end{align*}

Hence, $\gamma\leq \alpha$ as desired. Note that in the case $\gamma=0$ the formula written is not strictly correct (i.e, there is not $p^{\gamma-1}d$ correction term), but since $\alpha\geq 1$ our proposition is true a posteriori in this case. Now, we have that $Nc_N$ is a multiple of $p^{\alpha}$. In particular, $c_N$ must be a multiple of $p^{\alpha-\gamma}$. We use this to bound $N'$. Namely, from before we have the inequality

$$p^{\gamma}N'|c_N|\leq \sum_{d|N'}p^{\alpha}.$$

Seeing as $|c_N|\geq p^{\alpha-\gamma}$, we have equality if and only if $c_N=\pm p^{\alpha-\gamma}$ and $N'=\sum_{d|N'}1$. This means $N'=1$ or $N'=2$.

We now show that $N'=2$ is impossible. Let $\gamma$ now be the smallest value (if it exists) such that $c_{2p^{\gamma}}\neq0$. Then,

\begin{align*}
t_{2p^{\gamma}}&=\sum_{d|2p^{\gamma}}dc_d\\
&=\sum_{d|p^{\gamma}}dc_d+2p^{\gamma}c_{2p^{\gamma}}\\
&=t_{p^{\gamma}}\pm2p^{\alpha}.
\end{align*}

Since $t_{p^{\gamma}}$ and $t_{2p^{\gamma}}$ both equal $0$ or $p^{\alpha}$, it is impossible. Hence,  $c_{2p^{\gamma}}=0$ for all $\gamma$. This is the desired result.
\end{proof}

An important element to study in endomorphisms is $\left[x\right]$. We classify the possible places it can go:

\begin{proposition}
\label{Phi Classification}
For all $\phi \in \End\left(\MM_{\QQ}\right)$, we must have that $\phi\left(\left[x\right]\right)$ is one of the following: $[x]$, $0$, $[x-1]$, or $[(x-1)/x]$.
\end{proposition}
\begin{proof} As before, $t_k([x])$ takes on only the values $0$ or $1$. If $\phi([x])=c_0[x]+\sum_{n=1}^{\infty}c_n\left[x^n-1\right]$, then

$$\left|nc_n\right|=\left|\sum_{d|n}\mu(n/d)t_d(\phi([x]))\right|\leq \sum_{d|n}1$$

and so if $c_n \neq 0$ then $n\in \{0,1,2\}$. Moreover, in those cases we must have $c_n=1$ is in those cases. Looking over the finite number of options, we get the desired result
\end{proof}

The cases $\phi([x])=[x-1]$ or $\phi([x])=[(x-1)/x])$ can be treated with ease:

\begin{proposition}
\label{Degenerate phis}
The following properties of an endomorphism in $\End\left(\MM_{\QQ}\right)$ are equivalent:

\begin{enumerate}
\item $\forall$ $k\geq 1$, $\phi^*(k)=0$
\item $\exists$ $k\geq 1$ s.t $\phi^*(k)=0$
\item $\phi\left(\left[x\right]\right)=[x-1]$ or $[(x-1)/x]$
\item $\exists$ $k$ s.t $\phi(r)=t_0(r)[x-1]+\left(t_k(r)-t_0(r)\right)[x]$ $\forall$ $r\in \MM_{\QQ}$
\end{enumerate}
\end{proposition}
\begin{proof}$\,$

$(1)\implies(2)$. Trivial.

$(2)\implies(3)$. Let $k\geq 1$ be the value such that $\phi^*(k)=0$. Then, $t_k\left(\phi\left(\left[x\right]\right)\right)=t_{\phi^*(k)}\left(\left[x\right]\right)=1$. Hence, we cannot have $\phi\left(\left[x\right]\right)$ be $0$ or $[x]$ and so by Proposition \ref{Phi Classification} we get the result.

$(3)\implies(4)$. If $\phi([x])=[(x-1)/x]$, then

\begin{align*}
\phi\left([x]r\right)&=\phi\left(t_0(r)[x]\right)\\
&=t_0(r)\left[(x-1)/x\right]=t_0(r)[x-1]-t_0(r)[x]\\
&=\phi(r)[(x-1)/x]=\phi(r)-t_0(\phi(r))[x]
\end{align*}

so collecting we get that

$$\phi(r)=t_0(r)[x-1]+\left(t_{\phi^*(0)}(r)-t_0(r)\right)[x].$$

We proceed similarly when $\phi([x])=[x-1]$.

$(4)\implies(1)$. It is an unenlightening computation that every $\phi\in\End\left(\MM_{\QQ}\right)$ of the form in (4) is an endomorphism. For any $j\geq 1$,

$$t_j\left(t_0(r)[x-1]+\left(t_{\phi^*(0)}(r)-t_0(r)\right)[x]\right)=t_0(r)$$

so $\phi^*(j)=0$ and hence $(1)$ is implied. Combining these implications, we obtain our result.
\end{proof}

We now analyze the remaining case, $\phi^*(1)\neq 0$. The first step is to study the behavior of $\phi^*$ more closely.

\begin{lemma}
\label{vp dependence}
If $\phi \in \End\left(\MM_{\QQ}\right)$ with $\phi^*(1)\neq 0$, then $v_p(\phi^*(k))$ depends only on $v_p(k)$.
\end{lemma}
\begin{proof} By Proposition \ref{Degenerate phis}, $\phi^*(1)\neq0$ implies $\phi^*(k)\neq 0$ for all $k$ and hence $v_p(\phi^*(k))$ will all be well defined. By the classification given in Lemma \ref{MZZ Endos}, we get that $t_k(\phi([x^{p^{\alpha}}-1]))$ depends only on $v_p(k)$. In turn, $t_k(\phi([x^{p^{\alpha}}-1]))$ as $\alpha$ varies determines $v_p(\phi^*(k))$ and hence we are done.
\end{proof}

\begin{proposition}
\label{multiplicativity}
If $\phi^*(1)\neq0$, $\phi^*(n)/\phi^*(1)$ is multiplicative, and sends powers of $p$ to powers of $p$.
\end{proposition}
\begin{proof} If $n$ is a power of $p$, then $v_q(n)=v_q(n)$ for all $q\neq 0$, and hence $v_q(\phi^*(n))=v_q(\phi^*(1))$ for all $q\neq p$ and hence $\phi^*(n)/\phi^*(1)$ is a power of $p$. Thus, in general,

\begin{align*}
\phi^*(n)&=\prod_{p}p^{v_p(\phi^*(n))}\\
&=\prod_{p^{\alpha}|| n} p^{v_p(\phi^*(p^{\alpha}))} \cdot  \prod_{\substack{p^{\alpha}|| \phi^*(1) \\ p \nmid n}}p^{\alpha}\\
&=\phi^*(1) \cdot \prod_{p^{\alpha}|| n}p^{v_p(\phi^*(p^{\alpha}))-v_p(\phi^*(1))}
\end{align*}

and hence $\phi^*(n)/\phi^*(1)$ is multiplicative.
\end{proof}

\begin{lemma}
\label{first part}
Suppose $\phi^*(1)\neq0$, and $\phi\left(\left[x-k\right]\right)$ has a root of absolute value greater than $1$ for some $k\geq 2$. Then,

$$\limsup_{N\to\infty}\frac{\phi^*(n)}{n}$$

is a nonzero rational number.
\end{lemma}
\begin{proof} Let $C$ be the absolute value of the largest root of $r=\phi([x-k])$. By assumption, $C>1$. Letting $E<C$ be the absolute value of the second largest root of $\phi([x-k])$, we find that

\begin{align*}
k^{\phi^*(n)}&=t_{\phi^*(n)}\left([x-k]\right)\\
&=t_n\left(r\right)\\
&=S_n\cdot C^n+O\left(E^n\right),
\end{align*}

where $S_n=\sum_{j} c_j\theta_j^n$, $c_j\in \ZZ \neq 0$, and $|\theta_j|=1$ are all distinct. If $S_{N}$ were negative or complex for some $N$, then $S_N$ would be a limit point of $S_n$ by Lemma \ref{Periodic}. This is a contradiction since then we would have that $k^{\phi^*(n)}$ is negative or complex for some $n$. We can thus take logarithms, and get that

$$\phi^*(n)=\log(C)n +\log(S_n+O((E/C)^n)).$$

Here, by $\log$ we mean logarithm base $k$. It is standard that

$$\lim_{N\to\infty}\frac{1}{N}\sum_{n<N}|S_n|^2=\sum_{j}|c_j|^2.$$

Hence, $S_n\geq 1$ for a positive proportion of $n$ since $S_n$ is bounded above. In particular, we find that

$$\limsup_{n\to\infty}\frac{\phi^*(n)}{n}=\log(C)>0.$$

Choose $N\geq 1$ such that $S_N\neq 0$. Choose $q$ large such that

$$\left|S_n-S_{n+q}\right|\ll\frac{1}{q^{1/J}},$$

where $J\geq 1$ is a fixed integer. There are infinitely many such $q$ by Lemma \ref{Periodic}. We find uniformly over $m\ll q^{1/J}$

\begin{align*}
&\left|\phi^*(N+mq)-\log(C)\cdot mq\right|\\
\ll_N &\left|\phi^*(N+mq)-\log(C)\cdot (N+mq)+\log (S_{N+mq})\right|\\
+&\left|\log (S_{N+mq})-\log(S_N)\right|\\
\ll_N &1.
\end{align*}

We now manipluate to find that when $N^2m < q^{1/J}$,

\begin{align*}
&\frac{1}{\phi^*(N)}\left|\phi^*(N+(N^2m)q)-\log(C)\cdot (N^2m)q\right|\\
&=\left|\frac{\phi^*(1+m(Nq))}{\phi^*(1)}-\frac{N\log(C)}{\phi^*(N)}\cdot m(Nq)\right|\\
&\ll_N 1.
\end{align*}

Now, we observe the following. By Proposition \ref{multiplicativity}, $\phi^*(p^a)/\phi^*(1)$ is a power of $p$ for each prime $p$ and $a\geq 1$. If $\phi^*(p^a)/\phi^*(1) \geq p^{a+1}$, then $\phi^*(p^a)/p^a \geq \phi^*(1)p$.Since $\phi^*(n)\leq \log(C)n+O(1)$, this can only happen for finitely many $p$. In particular, let $m$ be the product of all primes for which $\phi^*(p^a)/\phi^*(1)\geq p^{a+1}$ for some $a\geq1$. Then $1+m(Nq)$ is relatively prime to all of these primes, and hence

$$\frac{\phi^*(1+m(Nq))}{\phi^*(1)(1+m(Nq))}$$

will be the reciprocal of an integer. Hence, since

$$\left|\frac{\phi^*(1+m(Nq))}{\phi^*(1)(1+m(Nq))}-\frac{N\log(C)}{\phi^*(N)}\right|\ll_N \frac{1}{q}$$

for large enough $q$, we have that $\frac{N\log(C)}{\phi^*(N)}$ can be approximated aribitarily well by reciprocals of integers. Seeing as the set of reciprocals of integers is discrete, this means that $\frac{N\log(C)}{\phi^*(N)}$ itself is the reciprocal of an integer. In particular, $\log(C)$ is rational and we get our result.
\end{proof}

\begin{lemma}
\label{second part}
In the language of Lemma \ref{first part}, $S_n=1$ for all $n$.
\end{lemma}
\begin{proof} By the discussion of Lemma \ref{first part}, $S_n\geq \sqrt{\sum_{j}|c_j|^2}$ for a positive proportion of $n$. Suppose that $\sum_{j}|c_j|^2=1$. Then, clearly, $S_n=\theta^n$ for a single value $\theta$. To make $S_n$ a positive real number for each $n$, we must have $\theta=1$. This gives the desired result.

The other case is that $\sum_{j}|c_j|^2\neq 1$. Here, we get that $S_n\geq \sqrt{2}$ for a positive proportion of $n$. For these values, we have

$$\phi^*(n)=\log(C)n+\log(S_n)+O((E/C)^n).$$

Let $D$ be the denominator of $\log(C)$, which exists by Lemma \ref{first part}. We have

$$D\phi^*(n)=D\log(C)n+D\log(S_n)+O((E/C)^n).$$

Since $S_n\geq \sqrt{2}$, $D\log(S_n)$ is bounded away from zero. Hence,

$$D\phi^*(n)=D\log(C)n+O(1),$$

where $O(1)$ is a bounded size non-zero integer. Every prime dividing $\phi^*(n)$ must either divide $\phi^*(1)$ or $n$ by Proposition \ref{multiplicativity}. If a prime divides both $\phi^*(n)$ and $n$ it must also divide $O(1)$, and hence be of bounded size. Hence, we get that there exists $B>0$ such that $\phi^*(n)$ is $B$-smooth for all $n$. The set of $B$-smooth numbers is of proportion zero in the naturals. In particular, $n/\log(C)$ is always a bounded distance from a $B$-smooth number and hence must be running over a set of proportion zero in the naturals. However, we assumed that $n$ was running over a set of positive density. This is a contradiction. Hence $S_n=1$ for all $n$ and we are done.
\end{proof}

\begin{proposition}
\label{make it mk lemma}

If $\phi^*(1)\neq 0$ and $\phi\left(\left[x-k\right]\right)$ has a root of absolute value greater than $1$ for some $k\geq 2$, then $\phi=\lambda_{\phi^*(1)}+(t_{\phi^*(0)}-t_0)[x]$.
\end{proposition}
\begin{proof} In the languages of Lemmas \ref{first part} and \ref{second part}, we have that

$$\phi^*(n)=\log(C)n+O((E/C)^n).$$

Seeing as $\phi^*(n)$ and $\log(C)n$ are rational, this means that $\phi^*(n)=\log(C)n$ for large enough $n$. By the multiplicativity of $\phi^*(n)/\phi^*(1)$, this means that $\log(C)=\phi^*(1)$, so $\phi^*(n)=\phi^*(1)n$ for large enough $n$. Again by the multiplicativity of $\phi^*(n)/\phi^*(1)$, this means that $\phi^*(n)=\phi^*(1)n$ for all $n\geq 1$. Introducing the  $\phi^*(0)=0$, and hence $\phi^*(n)=\phi^*(1)n$ for all $n\geq 1$.

It is clear that $\lambda_{\phi^*(1)}^{*}(n)=\phi^*(1)n$ for all $n\geq 0$. Introducing the $t_{\phi^*(0)}-t_0$ to get the correct value of $\phi^*(0)$, we conclude our proof.
\end{proof}

\begin{lemma}
\label{Finiteness of xp lemma}
If $\phi^*(1)\neq0$ and all the roots of $\phi([x-k])$ are of absolute value less than or equal to $1$ for some $k\geq 2$, then $\phi\left(\left[1-x^{-n}\right]\right)=0$ for all but finitely many $n$.
\end{lemma}
\begin{proof} We find that $t_k(\phi([x-k]))$ is bounded. Seeing as it is integral, we get that it ranges over finitely many values. Hence, $\phi^*(k)$ ranges over finitely many values. Thus, we have that $N\nmid \phi^*(k)$ for all $k\geq 1$, for all but finitely many $N$, by Proposition \ref{Degenerate phis}. For these $N$, we find that $t_k(\phi(\left[x^N-1\right]))$ is zero for all $k\geq 1$. Hence, we must have $\phi(\left[x^N-1\right])=N[x]$ for all but finitely many $N$, giving the result.
\end{proof}

\begin{theorem}
\label{Endomorphism Classification Theorem}
If $\phi\in \End(\MM_{\QQ})$, then either $\phi=\lambda_k+(t_j-t_0)$ for some $k,j$ or $\phi\left(\left[1-x^{-n}\right]\right)=0$ for all but finitely many $n$.
\end{theorem}
\begin{proof} If $\phi^*(1)=0$, then $\phi\left(\left[1-x^{-n}\right]\right)=0$ for all but finitely many $n$ by the $(2)\iff (3)$ equivalence in Proposition \ref{tk evaluation}. Since every $r\in \MM_\QQ$ can be written as an infinite sum of elements of the form $[1-x^{-n}]$ by Proposition [ref], this means that $\phi(r)$ will be an $N$th root of unity, where $N$ is the least common multiple of those $n$ such that $\phi([1-x^{-n}])$ is nonzero.

Alternatively, suppose $\phi^*(1)\neq 0$. If $\phi([x-2])\in \MM_{\ZZ,\Phi}$ then by Lemma \ref{Finiteness of xp lemma} we have that $\phi\left(\left[1-x^{-n}\right]\right)=0$ for all but finitely many $n$, and hence we can conclude by the above argument. Otherwise, by Lemma \ref{make it mk lemma} we have that $\phi$ is of the desired form.
\end{proof}

\section{Endomorphisms of $\MM_K$, $K/\QQ$ algebraic}

Let $\phi: \MM_K\to \MM_K$ be a continuous ring endomorphism. By Proposition \ref{Homs}, for every $n$ there exist $\sigma^{\phi}_n\in \Gal(K/\QQ)$ and $\phi^*(n)\in \ZZ_{\geq 0}$ so that $t_n \circ \phi= \sigma^{\phi}_n \circ t_{\phi^*(n)}$. In particular, the property of $t_n$ being rational for every $n$ is preserved by $\phi$, and hence in turn the property of having rational coefficients is preserved, by Newton's formulas.

Thus, $\phi$ restricts to a map $\MM_\QQ\to \MM_\QQ$. Theorem \ref{Endomorphism Classification Theorem} says then that either $\phi^*(k)$ is bounded uniformly with respect to $k$, or it is linear of the form $\phi^*(1)k$ for $k\geq 1$. We now treat these two cases:

\begin{proposition} Let $\phi: \MM_K\to \MM_K$ be a continuous ring endomorphism for which $\phi^*(k)$ is bounded. Then all roots of elements in the image of $\phi$ are roots of unity. If $K/\QQ$ is finite, then all the roots are $N$th roots for a fixed $N$.
\end{proposition}
\begin{proof} To begin, choose $r\in \MM_K$. Then $t_k(\phi(r))=\sigma_k (t_{\phi^*(k)}(r))$ runs over finitely many values, since $t_{\phi^*(k)}(r)$ runs over finitely many values and $\Gal(K/\QQ)$ can send those values to some subset of their (finitely many) Galois conjugates. By Proposition \ref{finite periodic}, we get that all the roots of $\phi(r)$ are roots of unity.

Suppose $K/\QQ$ is finite. Then, $t_{k}(\phi(r))$ goes over finitely many values, and that finite quanity is now uniform as $r$ varies. Suppose there were no $N$ such that every $\phi(r)$ were periodic mod $N$. Then, by taking weighted linear combinations of these values we could create elements $r$ for which $\phi(r)$ takes on arbitrarily many values. This is a contradiction, hence $\phi(r)$ is periodic mod $N$ for some uniform period $N$.

We now show that $r$ being periodic mod $N$ implies all of its nonzero roots are $N$th roots of unity, which in turn will imply our result. Write $t_k(r)=\sum_i c_i \alpha_i^k$. Seeing as this is periodic mod $N$, we find that

$$\sum_{i}c_i \alpha_i^{k_0}\left(\alpha_i^N\right)^k$$

is constant with respect to $k$ for every $k_0$. Hence, we find that

$$\sum_{\alpha_i^N=\alpha}c_i \alpha_i^{k_0}=0$$

for all $\alpha\neq 1$. Thus, again by Proposition \ref{finite periodic}, we get that $c_i=0$ whenever $\alpha_i^N\neq 1$. In particular, all $\alpha_i$ are $N$th roots of unity, so we are done.
\end{proof}

\begin{lemma}
\label{preperation lemma}
Let $\alpha\in \overline{\QQ}$ be an algebraic number. For all large enough integers $N$, none of the Galois conjugates of $\alpha'=N\alpha+1$ differ in ratio by a root of unity.
\end{lemma}
\begin{proof}
Suppose that there are infinitely many $N$ for which $N\alpha+1$ and $N\sigma(\alpha)+1$ differ in ratio by a root of unity, for some automorphism $\sigma$. Then there is an infinite sequence $\frac{\alpha+1/N}{\sigma(\alpha)+1/N}$ of roots of unity. All of these terms have bounded algebraic degree over $\QQ$, and hence they must all be $J$th roots of unity for some $J$. The set of $J$th roots of unity is discrete, and hence a non-constant convergent sequence of $J$th roots of unity is impossible. Thus, we get a contradiction and the desired result is obtained.
\end{proof}

\begin{proposition} Let $\phi:\MM_K\to \MM_\CC$ be a continuous ring endomorphism for which $\phi^*(n)=\phi^*(1)n$ for all $n\geq 1$. Then, $\sigma_n^{\phi}=\sigma_1^{\phi}$ for all $n\geq 1$.
\end{proposition}
\begin{proof} Choose $\alpha \in K$. We wish to show that $\sigma^\phi_n(\alpha)=\sigma^\phi_1(\alpha)$. By Proposition \ref{preperation lemma}, we can assume without loss of generality no Galois conjugates of $\alpha$ differ in ratio by root of unity. Considering where $\phi([x-\alpha])$ goes, we find that

$$\sigma^{\phi}_n(\alpha)^{\phi^*(1)n}=t_n(\phi([x-\alpha]))=S_n\cdot C^n+O(E^n),$$

where $C$ is a a real number, $E<C$, and $S_n=\sum_i c_i\cdot \theta_i^n$ is a weighted sum of complex numbers of the unit circle. Taking absolute values and a logarithm, we get that uniformly for $n$ such that $|S_n|\geq 6/5$

$$\log|\sigma_n^{\phi}(\alpha)|\cdot \phi^*(1)n=\log(C)\cdot n+\log|S_n|+O((E/C)^n).$$

Taking ratios, we get

$$\log|\sigma_n^{\phi}(\alpha)|=\frac{\log(C)}{\phi^*(1)}+\frac{\log|S_n|}{\phi^*(1)n}+O((E/C)^n).$$

Seeing as $\sigma_n^{\phi}(\alpha)$ go over a finite set as $n$ varies, the only way to be arbitrarily well approximated by $\log(C)/\phi^*(1)$ is if $\log|\sigma_n^{\phi}(\alpha)|=\frac{\log(C)}{\phi^*(1)}$ for large enough $n$. Hence, for large enough $n$ with $|S_n|\geq 6/5$ we have

$$\frac{\log|S_n|}{\phi^*(1)n}=O((E/C)^n).$$

Seeing as $\log |S_n|$ is bounded away from $0$, this can only happen for finitely many $n$. Hence, $|S_n|\geq 6/5$ for only finitely many $n$. Hence,

$$\lim_{N\to\infty}\frac{1}{N}\sum_{n=1}^{N}|S_n|^{2}=\sum_{i}|c_i|^2\leq (6/5)^2<2.$$

Seeing as the $|c_i|$ are integers, for $\sum_i |c_i|^2$ to be less than $2$ we must have $c_i=1$ for some $i$, and $c_i=0$ for all other $i$. Hence, $S_n=\theta^n$ for some $|\theta|=1$. In particular, we find that

$$\sigma^{\phi}_n(\alpha)^{\phi^*(1)n}=(\theta C)^{n}+O(E^n)$$

for all large enough $n$. Hence,

$$\sigma^{\phi}_n(\alpha)^{\phi^*(1)n}=(\theta C)^{n}$$

for all $n\geq 1$. Hence,

$$\left(\frac{\sigma_n^{\phi}(\alpha)}{\sigma_1^{\phi}(\alpha)}\right)^{\phi^*(1)n}=1$$

for all $n\geq 1$. By our assumption on $\alpha$, we thus get that $\sigma_n^{\phi}(\alpha)=\sigma_1^{\phi}(\alpha)$. We get the desired conclusion.
\end{proof}

Combining the two cases, we get the desired result.

\end{document}